\documentclass[11pt,a4paper, envcountsame]{amsart}
\usepackage[usenames,dvipsnames]{color}

 \usepackage[colorlinks,citecolor=blue,urlcolor=red]{hyperref}
 
 \usepackage{amsmath, amssymb, xspace}
 \usepackage{graphicx}

 \usepackage[all]{xy}

\usepackage{tikz}
\usetikzlibrary{arrows,snakes,positioning,backgrounds,shadows}

\newtheorem{theorem}{Theorem}[section]

\newtheorem{fact}[theorem]{Fact}

\newtheorem{remark}[theorem]{Remark}

\newtheorem{lemma}[theorem]{Lemma}

\newtheorem{cor}[theorem]{Corollary}

\newtheorem{question}[theorem]{Question}

\theoremstyle{definition}

\newtheorem{df}[theorem]{Definition}

\newcommand{\NN}{{\mathbb{N}}}
\newcommand{\RR}{{\mathbb{R}}}
\newcommand{\QQ}{{\mathbb{Q}}}

\newcommand{\sub}{\subseteq}
\newcommand{\sN}[1]{_{#1\in \NN}}

\newcommand{\bi}{\begin{itemize}}
\newcommand{\ei}{\end{itemize}}
\newcommand{\bc}{\begin{center}}
\newcommand{\ec}{\end{center}}

\newcommand{\estring}{\emptyset}

\newcommand{\ex}{\exists}
\newcommand{\fa}{\forall}

\newcommand{\aaa}{\alpha}

\newcommand{\wt}{\widetilde}
\newcommand{\ol}{\overline}

\newcommand{\vectornorm}[1]{\left| \! \left|#1\right| \! \right|}

\newcommand{\maxnorm}[1]{{\vectornorm {#1} }_\text{max}}

\begin{document}

\title{The Scott rank of Polish metric spaces}

\author{Sy Friedman, Martin K\"orwien, Ekaterina Fokina, Andr\'e Nies}
\begin{abstract} We study the usual notion of Scott rank but in the setting of   Polish metric spaces. The signature  consists of distance relations: for each rational $q > 0$, there is a relation $R_{<q}(x,y)$  stating  that the distance of $x$ and $y $ is less than $q$. We show that compact spaces have  Scott rank at most $\omega$, and that there are discrete ultrametric spaces of arbitrarily high countable Scott rank.

This research was carried out in 2012 and has been available on Nies' publications web page.  We are putting it up on arXiv as a reference because of later research on the open question posed in here whether the Scott rank of a Polish metric space is always countable.  We include updates on the question.
\end{abstract}
\maketitle

\section{Introduction}

\begin{df} \label{sig} We view a metric space $(X,d)$ as a    structure for the signature
 \bc  $\{R_{q}\colon\, q \in \QQ^+\}$, \ec  
 where the $R_{q}$  are  binary relation symbols. The intended meaning of $R_{ q} x y$ is that $d(x,y)< q$.  
 \end{df}  Clearly, isomorphism of such structures  is the same as  isometry.

We ask to which extent  Polish metric spaces behave like countable structures.  For instance, just like for countable structures, if $X_0 \cong_p X_1$  (partial isomorphism, given by system of back and forth relations as in  \cite{Barwise:73}) then $X_0 \cong X_1$.

For the reader's convenience, we recall the
definitions of $\alpha$-equivalence and
Scott rank of a structure.

\begin{df}
	Let $M$ be an $\mathcal{L}$-structure.
	We define inductively what it means for
	finite tuples of same length 
	$\bar{a},\bar{b}$ from $M$ 
	to be $\alpha$-equivalent, denoted by
	$\bar{a}\equiv_\alpha\bar{b}$.
	\begin{itemize}
		\item $\bar{a}\equiv_0\bar{b}$ if and
		only if the quantifier-free types of
		the tuples are the same.
		\item For a limit ordinal $\alpha$,
		$\bar{a}\equiv_\alpha\bar{b}$ if and
		only if $\bar{a}\equiv_\beta\bar{b}$
		for all $\beta<\alpha$.
		\item $\bar{a}\equiv_{\alpha+1}\bar{b}$
		if and only if both of the following hold:
		\begin{itemize}
			\item For all $x\in M$, there is some
			$y\in M$ such that 
			$\bar{a}x\equiv_\alpha\bar{b}y$
			\item For all $y\in M$, there is some
			$x\in M$ such that 
			$\bar{a}x\equiv_\alpha\bar{b}y$
		\end{itemize}
	\end{itemize}
\end{df}

The \emph{Scott rank} $\mathrm{sr}(M)$ of a 
structure $M$ is defined as the smallest 
$\alpha$ such that $\equiv_\alpha$
implies $\equiv_{\alpha+1}$ for all tuples 
of that structure. We remark that always
$\mathrm{sr}(M)<|M|^+$.

%
%
A metric  space $X$ is called ultrahomogeneous if any isometry between finite  subsets of $X$ can be extended to an  auto-isometry of $X$. 
 Using a back and forth argument, one shows that a countable structure has Scott rank $0$ iff it is ultrahomogeneous. 
\begin{fact} A Polish space has Scott rank $0$ iff it is ultrahomogeneous.  \end{fact}
The nontrivial,   left-to right direction is proved via a back and forth argument where the $\forall $ player takes points from the countable dense set. The resulting partial isometry can be extended to a full isometry by completeness.

\begin{question} \ 
\bi \item[(a)] Does every Polish metric space $M$ have  countable Scott rank? 

\item[(b)] Can    every Polish metric space  be described within the class of Polish spaces by  an  $L_{\omega_1, \omega} $     sentence?  \ei \end{question}


By analogy with countable models, some evidence for an affirmative answer to  (b) can  be obtained as follows. By Gao and Kechris~\cite{Gao.Kechris:03}, isometry of Polish spaces (suitably encoded by a real) is $\le_B$ the orbit equivalence relation $E_I$ obtained by the action of $Iso(\mathbb U)$ on the Effros Space $F(\mathbb U)$.  Each orbit is Borel (Luzin-Nardzewsky). So each isometry class of Polish spaces is Borel.

\begin{remark}[May 2019] 

Doucha \cite{Doucha:14, Doucha:17} proved that the Scott rank of $M$ is at most $\omega_1$. Whether it is always countable remains open.

 Ben Yaacov et al \cite[Section 10]{BenYaacov.Nies.ea:17}  answer part (b) in the affirmative.  They produced a Scott sentence in infinitary continuous logic, and observed that it can be translated to classical infinitary logic.

\end{remark}

%
%

Can we show directly that each   isometry class is Borel?

\begin{remark}[May 2019]  Ben Yaacov et al \cite{BenYaacov.Nies.ea:17} have given  a different argument  that each  isometry class not relying on the descriptive set theory approach of~\cite{Gao.Kechris:03}.  Rather,  they construct a Scott sentence in infinitary continuous logic, and the models of such a sentence always form a Borel set. This is not the case for $L_{\omega_1, \omega} $ sentences in classical logic.
\end{remark}
One can verify  that the Scott rank of a metric structure is less than the least hyperprojective $\alpha$; that is, the least $\alpha$ such that $L_\alpha (\RR) $ is a model of Kripke-Platek set theory. This is in fact true for Borel structures.

 There is a Borel linear order $S$ with Scott rank $\omega_1$: reals $b \in [0,1]$  encode binary relations $R_b$. Let $L = \{ b \in [0,1] \colon R_b \text{ is linear}\}$. Let  $S$ be the sum of all $R_b$ for $b \in L$. 

\section{Compact metric spaces are $\exists$-homogenous and hence have Scott rank at most $\omega$.}

Let $X$ be a metric space.  For a tuple $\ol x \in X^n$  consider  the $n\times n$ distance matrix \[D_n(\ol x) = d(x_i,x_j)_{i,j<n}. \] We often   view this matrix as a tuple in $\RR^{n^2}$ with the max norm $\maxnorm{.}$.   Note that for any matrix  $A \in \QQ^{n^2}$ and any positive rational $p$, there is a quantifier-free positive  first-order formula $\phi_{A,n, p} (\ol x)$ in the metric signature (Definition~\ref{sig}) expressing that $\maxnorm { D_n(\ol x) - A} < p$.

 Let $  D_n (X) \sub\RR^{n^2}$ denote the set of $n\times n$ distance matrices.
Let now  $X$ be a compact metric space. Then  for  each $n$, the set $  D_n (X) $ is compact. Gromov \cite{Gromov:07} showed that the space $X$  is described up to isometry by this sequence of compact sets  $D_n (X)\sN n$ (see \cite[proof of  14.2.1]{Gao:09}, but note that our $D_n$ is denoted $M_{n-1}$ there). This shows that isometry of compact spaces is smooth, i.e., Borel reducible to the  identity on $\RR$.  

An \emph{existential positive formula} is a first-order formula   not using $\forall, \lnot$.

 \begin{theorem}   Let $X, Y$ be compact metric spaces. Suppose that   tuples $\widetilde a \in X^p, \widetilde b \in Y^p$  satisfy  the same existential  positive formulas. Then there is an isometry from $X$ to $Y$ mapping $\widetilde a$ to $\widetilde b$.  In particular, each compact metric space  $(X,d)$ is $\exists$-homogeneous. \end{theorem}
 
 \begin{proof}  Recall that any isometric self-embedding of a compact metric space is onto (see \cite[proof of  14.2.1]{Gao:09}). So  by the symmetry it suffices to find an isometric  embedding of $X$ into $Y$  mapping $\widetilde a$ to $\widetilde b$.
 
 The following slightly extends the above-mentioned result of   Gromov (see \cite[Exercise 14.2.3]{Gao:09}).
 \begin{lemma}  Suppose that for any $\epsilon >0$, for any   $n$ and tuple $\ol x \in X^{n}$ there is a tuple $\ol y \in Y^{n}$ such that \[\vectornorm{D(\widetilde a , \ol x ) - D(\widetilde b,  \ol y) }_{max} < \epsilon.\] Then there is an  an isometric  embedding of $X$ to $Y$  mapping $\widetilde a$ to $\widetilde b$.
 \end{lemma}
 
 \begin{proof} For each $k$ let $\ol y_k$ be a tuple such that $\vectornorm{D(\widetilde a , \ol x ) - D(\widetilde b,  \ol y) }_{max} < 1/k$. Let $\ol y $ be the limit of a convergent subsequence of $(\ol y_k)\sN k$. This shows that for each $\ol x $ there is $\ol y $ such that   $D(\widetilde a , \ol x )  = D(\widetilde b,  \ol y)$. 
 
 We can now proceed almost exactly as in \cite[proof of  14.2.1]{Gao:09}. Let $(x_i)\sN i$ be a dense sequence in $X$. For each $n\in \NN$ there are $y^n_0, \ldots, y^n_n\in Y$ such that 
 \bc $D(\widetilde a , x_0, \ldots, x_n )  = D(\widetilde b,   y^n_0, \ldots, y^n_n)$. \ec
 There is $A_0 \subset \NN$  with $0 \not \in A_0$  such that $z_0:= \lim_{n \in A_0} y^n_0$ exists.  There is $A_1 \subset A_0$ with $\min A_0 \not \in A_1$ such that $z_1:= \lim_{n \in A_1} y^n_1$ exists. Proceeding that way we obtain a sequence of points $z_n \in Y$ and a descending sequence of sets $A_0 \supset A_1 \supset A_2 \supset \ldots$ with $\min A_k \not \in  A_{k+1}$ such that $z_k = \lim_{n \in A_k} y^n_k$.
 
 Let $A = \{ \min A_k \colon \, k \in \NN\}$. Then $A \setminus A_k$ is finite for each $k$, so $z_k = \lim_{n \in A} y^n_k$ for each $k$. One uses this to show that $d(x_i,x_j ) = d(z_i, z_j)$ for each $i,j$, and $d(a_r, x_i) = d(b_r, z_i)$ for each $r, j$. Hence the map $x_i \mapsto z_i$ can be extended to the required isometric embedding of $X$  into $Y$  
 \end{proof}
 
 It now suffices to show that if $\widetilde a \in X^p, \widetilde b \in Y^p$  satisfy  the same existential positive formulas, then the hypothesis of the lemma is satisfied.  Recall that the formula $\phi_{A,n,p} (\ol x)$ expresses that $\maxnorm { D_n(\ol x) - A} < p$. 
 Given $\ol  x \in  \ X^{n}$ choose a rational $(k+n ) \times (k+n)$ matrix $A$ such that 
 $$\maxnorm{D(\widetilde a, \ol x) - A} < p= \epsilon/2.$$
 Thus $\exists \ol x  \, \phi_{A,n+k,p}(\widetilde a , \ol x)$ holds in $ X$. Hence  there is $ \ol y \in Y^n$ such that  $  \phi_{A,n+k, p}(\widetilde b , \ol y)$ holds in   $Y$. This implies $\vectornorm{D(\widetilde a , \ol x ) - D(\widetilde b,  \ol y) }_{max} < \epsilon$ as required.  
 \end{proof}
 
 The case $\widetilde a = \widetilde b = \estring$ yields:
 \begin{cor} Within the class of compact metric spaces, each member is uniquely described by its existential positive first-order  theory. \end{cor}

 Note that a complete metric space is compact iff it is totally bounded, namely, satisfies the $L_{\omega_1, \omega} $ sentence 
 
 \[\bigwedge_{q \in \QQ^+} \bigvee_{n \in \NN} \ex x_0 \ldots x_{n-1} \fa y \bigvee_{i< n} d(x_i,y)< q. \]
 Thus   using the corollary above,  each compact metric space can be described by an   $L_{\omega_1, \omega} $ sentence within the class of Polish metric  spaces.
 
 \begin{cor} The Scott rank of any compact metric space $X$ is at most~$\omega$. \end{cor}
 
 \begin{proof} In the notation of the theorem with $Y= X$, suppose $\wt a \equiv_n \wt b$ for each $n \in \omega$. Then $\wt a$ and $\wt b$ satisfy the same existential formulas. Hence there is an isometry of $X$ sending $\wt a$ to $\wt b$. \end{proof} 
 
%
%
%
%
%
\section{Countable Polish spaces can have arbitrary countable Scott rank	}

Countable structures  have countable
Scott rank. In this section, we show that
arbitrary large countable ordinal ranks are 
possible for countable Polish metric spaces. The examples
we construct  actually will be discrete.

We will inductively construct sub-trees of
$\omega^{<\omega}$ with growing Scott
ranks. Note that we can view such  trees as model
theoretic structures, for example in the language 
of a unary function $f$ with the semantics
``$f(a)=b$ if and only if $b$ is the immediate 
predecessor of $a$''. Also, those trees can be
regarded as metric spaces with the metric induced
by the usual one on $\omega^{<\omega}$.
It is easy to see that the relations $\equiv_\alpha$
(and hence the notion of Scott rank) is the
same in both the tree structure and the
metric space structure.

We denote finite sequences from
$\omega^{<\omega}$ by $\langle n_1,n_2,
\dots,n_k\rangle$ and the concatenation
of two sequences $s,t$ by $s^\frown t$. 
Now we proceed to the
definition of $T_n^\alpha\subset
\omega^{<\omega}$ ($n\leq\omega$,
$\alpha<\omega_1$) by induction over
$\alpha$.

\begin{df}
	For any $n\leq\omega$, let $T_n^0$
	be the collection of the empty sequence
	and the sequences $\langle a\rangle$
	for all $a<n$.
	
	Now suppose we have defined all the
	$T_n^\alpha$ for some $\alpha<\omega_1$.
	Let $T_n^{\alpha+1}$ be the collection of
	\begin{itemize}
		\item the empty sequence
		\item all sequences
		of the form $\langle 2a\rangle^\frown s_a$
		where $a<\omega$ and 
		$s_a\in T_a^\alpha$
		\item for any $a<n$, all sequences
		of the form $\langle 2a+1\rangle^\frown s$
		with $s\in T_\omega^\alpha$
	\end{itemize}
	
	Finally, suppose $\alpha$ is a countable
	limit ordinal and that the $T_n^\beta$
	are already defined for all $n\leq\omega$
	and $\beta<\alpha$. Fix any bijection
	$B:\alpha\times\omega\rightarrow
	\omega$ and define $T_n^{\alpha}$
	as the collection of
	\begin{itemize}
		\item the empty sequence
		\item all sequences
		of the form $\langle a\rangle^\frown s_a$
		where $a=B(c,d)$ for some $c<\alpha$
		and $d\leq n$ and $s_a\in T_\omega^c$
	\end{itemize}
\end{df}

Informally speaking, at each stage of the
construction we glue infinitely many
of the already defined trees into a new one. 
$T_n^{\alpha+1}$ contains all $T_k^{\alpha}$
for $k<\omega$ exactly once and
$T_\omega^{\alpha}$ exactly $n$ times.
At limit stages $\aaa$, we amalgamate $n+1$ copies
of each $T_\omega^\beta$ ($\beta<\alpha$).

Inductively, we see that none
of the constructed trees has an
infinite path and that all of them,
seen as metric spaces, are discrete
and complete (thus they are Polish).
The $T_m^0$ are homogeneous, so they
have Scott rank $0$, and we can verify
inductively that $T_n^\alpha$ has
Scott rank $\alpha\cdot\omega$ for
all $\alpha<\omega_1$ and all 
$n\leq\omega$.
\def\cprime{$'$}


\begin{thebibliography}{1}

\bibitem{Barwise:73}
Jon Barwise.
\newblock Back and forth through infinitary logic.
\newblock pages 5--34. MAA Studies in Math., Vol. 8, 1973.

\bibitem{BenYaacov.Nies.ea:17}
I.~Ben~Yaacov, M.~Doucha, A.~Nies, and T.~Tsankov.
\newblock Metric {S}cott analysis.
\newblock {\em Advances in Mathematics}, 318:46--87, 2017.

\bibitem{Doucha:14}
M.~Doucha.
\newblock Scott rank of {P}olish metric spaces.
\newblock {\em Annals of Pure and Applied Logic}, 165(12):1919 -- 1929, 2014.

\bibitem{Doucha:17}
M.~Doucha.
\newblock Erratum to: {S}cott rank of {P}olish metric spaces [{A}nn. {P}ure
  {A}ppl. {L}ogic 165 (12) (2014) 1919-€"1929].
\newblock {\em Annals of Pure and Applied Logic}, 168(7):1490, 2017.

\bibitem{Gao:09}
Su~Gao.
\newblock {\em Invariant descriptive set theory}, volume 293 of {\em Pure and
  Applied Mathematics (Boca Raton)}.
\newblock CRC Press, Boca Raton, FL, 2009.

\bibitem{Gao.Kechris:03}
Su~Gao and Alexander~S. Kechris.
\newblock On the classification of {P}olish metric spaces up to isometry.
\newblock {\em Mem. Amer. Math. Soc.}, 161(766):viii+78, 2003.

\bibitem{Gromov:07}
M.~Gromov.
\newblock {\em Metric structures for Riemannian and non-Riemannian spaces}.
\newblock Springer Science \& Business Media, 2007.

\end{thebibliography}
\end{document}